\documentclass[11pt,reqno,a4paper]{amsart}
\usepackage[latin1]{inputenc}
\usepackage[latin1]{inputenc}
\usepackage{a4wide}
\usepackage{fancyhdr}
\usepackage{amsfonts,amssymb,amsmath,amsthm,latexsym,indentfirst}
\usepackage{epsfig}
\usepackage{enumerate}
\usepackage{amsfonts}
\usepackage{latexsym}
\usepackage{amssymb}
\usepackage{amsmath}
\usepackage{textcomp}
\usepackage{graphicx}

\newtheorem{definition}{Definition}
\newtheorem{theorem}[definition]{Theorem}
\newtheorem{lemma}[definition]{Lemma}
\newtheorem{proposition}[definition]{Proposition}
\newtheorem{remark}[definition]{Remark}

\newcommand{\N}[1][]{\ensuremath{{\mathbb{N}^{#1}} }}

\newcommand{\re}{\mathbb{R}}

\newcommand{\ce}{\mathbb{C}}

\def\Re{ \mathrm{Re}\, }

\title{The Davey Stewartson system in Weak $L^p$ Spaces}
\date{\today}
\author[Vanessa Barros]{}
\email{vbarros@impa.br}
\begin{document}
\maketitle
\centerline{\scshape Vanessa Barros}
\medskip
{\footnotesize
% please put the address of the first author
 \centerline{Universidade Federal de Alagoas, Instituto de matem\'atica}
  \centerline{57072-090, Macei\'o, Alagoas, Brasil }
} 
\begin{abstract}
We study the global Cauchy problem associated to the Davey-Stewartson system in $\re^n,\ n=2,3$. Existence and uniqueness 
of solution are stablished
for small data in some weak $L^p$ space. We apply an  interpolation theorem and the 
generalization of the Strichartz estimates for the Schr\"odinger equation derivated in \cite{CVeVi}. 
As a consequence we obtain self-similar
solutions.

\end{abstract}

\section{Introduction}
This paper is concerned with the initial value problem (IVP) associated to the Davey-Stewartson system

\begin{equation}\label{eqds1}
\left\{\begin{array}{l}
i \partial_t u + \delta\partial_{x_1}^2 u+\sum_{j=2}^{n}{\partial_{x_j}^2u}  = \chi |u|^{\alpha}
 u +bu\partial_{x_1} \varphi,  \\
  \partial_{x_1}^2\varphi+ m\partial_{x_2}^2 \varphi+\sum_{j=3}^{n}{\partial_{x_j}^2\varphi}  = 
 \partial_{x_1}(|u|^{\alpha}),\\
u(x,0)  = u_0(x)
\end{array}
\right.\quad  (x, t) \in \re^n\times \re,\ n=2,3,
\end{equation}
\vspace{0.4cm}
\noindent \noindent
where $u=u(x,t)$ is a complex-valued function and  $\varphi=\varphi(x,t)$ is a real-valued function.

The exponent $\alpha$ is such that $\frac{4(n+1)}{n(n+2)}<\alpha<\frac{4(n+1)}{n^2}$, the parameters
$\chi$ and $ b$ are constants in $\re$, $\delta$ and  $m$ are real positive and we can consider $\delta, \chi$
 normalized in
such a way that $|\delta|=|\chi|=1$.

The Davey-Stewartson systems are  $2D$ generalization of the cubic $1D$ Schr\"odinger equation, 
$$i\partial_t u+\Delta u=|u|^2u$$
 and model
the evolution of weakly nonlinear water waves that travel predominantly in one direction but  which the
 amplitude is modulated slowly
 in two horizontal directions. 

 System \eqref{eqds1}, $n=2,\ \alpha=2$, was first derived for Davey and Stewartson \cite{DS} in the
 context of water waves,
but its analysis did not take account of the effect of surface tension (or capillarity). This effect 
was later included
by Djordjevic and Redekopp \cite{DR} who have shown that the parameter $m$ can become negative when
 capillary effects are important. 
Independently, Ablowitz and Haberman \cite{AH}
obtained a particular form of \eqref{eqds1}, $n=2$, as an example of completely integrable model also
 generalizing the two-dimensional 
nonlinear Schr\"odinger equation.

When $(\delta,\chi, b,m)=(1,-1,2,-1),(-1-2,1,1),(-1,2,-1,1)$ the system \eqref{eqds1}, $n=2$, is referred as 
$DSI, DSII$ defocusing and $DSII$ focusing respectively in the inverse scattering literature. In these cases
several results concerning the existence of solutions or lump solutions have been 
established (\cite{AF}, \cite{AnFr}, \cite{AS}, \cite{C}, \cite{FS}, \cite{FSu}, \cite{Su}) by the inverse scattering 
techniques.

In \cite{GS}, Ghidaglia and Saut studied the existence of solutions of IVP \eqref{eqds1}, $n=2,\ \alpha=2$. 
They classified the system as 
elliptic-elliptic, elliptic-hyperbolic, hyperbolic-elliptic and hyperbolic-hyperbolic, according to respective 
sign of $(\delta,m):(+, +),
 (+, -), (-, +), (-, -).$

For the elliptic-elliptic and hyperbolic-elliptic cases, Ghidaglia and Saut \cite{GS} reduced the system 
\eqref{eqds1}, $n=2$,
 to the nonlinear
 cubic Schr\"odinger equation with a nonlocal nonlinear term, i.e.
$$i \partial_t u + \delta\partial_{x_1}^2 u +\partial_{x_2}^2u =  \chi |u|^2 u +H(u),$$
where $H(u)=(\Delta^{-1}\partial_x^2|u|^2)u.$ They showed local well-posedness for data in $L^2, H^1$ and $H^2$ using
Strichartz estimates and the continuity properties of the operator $\Delta^{-1}$.

The remaining cases, elliptic-hyperbolic and hyperbolic-hyperbolic, were
treated by Linares and  Ponce \cite{LP1}, Hayashi \cite{H1}, \cite{H2}, Chihara \cite{Ch}, Hayashi and Hirata 
\cite{HH1}, \cite{HH2}, Hayashi 
and Saut \cite{HS}(see \cite{LP2} for further references).

Here we will concentrate in the elliptic-elliptic and hyperbolic-elliptic cases. We start with the motivation for this work:

From the condition $m>0$ we are allowed to reduce the Davey-Stewartson system  \eqref{eqds1} to the Schr\"odinger
equation

\begin{equation}\label{eqs5}
\left\{\begin{array}{l}
i \partial_t u + \delta\partial_{x_1}^2 u +\sum\limits_{j=2}^{n}{\partial_{x_j}^2u}  = 
 \chi |u|^{\alpha} u +bu E(|u|^{\alpha}), \\
 u(x,0)  = u_0(x),
\end{array}
\right. \qquad\qquad \forall\, x \in \re^n ,\ n=2,3,\ t \in \re,
\end{equation} 

\vspace{.2cm}
\noindent \noindent
where
\begin{equation}\label{defE1}
 \widehat{E(f)}(\xi)=\frac{\xi_1^2}{{\xi_1^2}+m\xi_2^2+\sum_{j=3}^{n}{\xi_j^2}} \hat{f}(\xi)=p(\xi)\hat{f}(\xi).
\end{equation}

%Using Strichartz estimates to the Schr\"odinger equation we deduce some inequalities that will be the key to run the 
%fixed point
% argument and prove existence and uniqueness of solutions 
%in some weak $L^p$ spaces.

Now observe that if $u(x, t)$ satisfies 
$$i u_t + \delta\partial_{x_1}^2 u +\sum_{j=2}^{n}{\partial_{x_j}^2u}  =  \chi |u|^\alpha u +bu E(|u|^\alpha),$$
 then  also does $u_\beta(x, t)=\beta^{2/\alpha} u(\beta x, \beta^2 t)$,
 for all $\beta>0.$ 

Therefore it is natural to ask whether solutions $u(x, t)$ of \eqref{eqds1}
exist and satisfy, for $\beta>0$:
$$u(x, t)=\beta^{2/\alpha} u(\beta x, \beta^2 t).$$

Such solutions are called self-similar solutions of the equation \eqref{eqs5}.
% Formally:
%\begin{definition}
%u(x, t) is said to be a self-similar solution to the Schr\"odinger equation in \eqref{eqs5} if
%$$u(x, t)=u_\beta(x, t)=\beta^{2/\alpha} u(\beta x, \beta^2 t), \ \forall\  \beta>0.$$
%\end{definition}

Therefore supposing local well posedness and $u$ a self-similar solution we must have
$$u(x,0)=u_\beta(x, 0),\ \forall \ \beta>0,$$
i.e.,
$$u_0(x)=\beta^{2/\alpha} u_0(\beta x).$$

 In other words, $u_0(x)$ is homogeneous with degree $-2/\alpha$ and every initial data that gives a 
self-similar solution must verify this property. Unfortunately, those functions do not belong to the usual spaces
where strong solutions exists, such as the Sobolev spaces $H^s(\re^n)$.
We shall therefore replace them by other functional spaces that allow homogeneous functions.

There are many motivations to find self-similar solutions.
One of them is that they  can  give a good description of the large time behaviour for solutions
of dispersive equations.

The idea of constructing self-similar solutions by solving the initial value problem for homogeneous data was first used by 
Giga and Miyakawa \cite{GM}, for the Navier Stokes equation in vorticity form. The idea of \cite{GM} was used latter by
Cannone and Planchon \cite{CP}, Planchon \cite{P} (for the Navier-Stokes equation);  Kwak \cite{K}, Snoussi,
Tayachi and Weissler \cite{STW} (for nonlinear parabolic problems); Kavian and Weissler \cite{KW}, Pecher \cite{Pe}, 
 Ribaud and Youssfi \cite{RY2} (for the nonlinear wave equation); Cazenave and Weissler \cite{CW1},\cite{CW2}, Ribaud 
and Youssfi
\cite{RY1}, Furioli \cite{F} (for the nonlinear Schr\"odinger equation).

In  \cite{CVeVi} Cazenave, Vega and Vilela studied the global Cauchy problem for the Schr\"odinger equation 

\begin{equation}\label{eqschr}
 i \partial_t u + \Delta u  =  \gamma |u|^\alpha u, \ \ \ \ \ \alpha>0,\ \gamma \in \re, \ (x,t) \in \re^n
 \times [0, \infty).
\end{equation}

Using a generalization of the  Strichartz estimates for the Schr\"odinger equation they showed that, under 
some restrictions on $\alpha$, if the initial value is sufficiently small in some weak $L^p$ space then there exists 
a global solution. 
This result provided a common framework to the classical $H^s$ solutions and to the self-similar solutions. We follow their 
ideas in our work. From the condition $m>0$ we are allowed to reduce the Davey-Stewartson system  \eqref{eqds1} to the
 Schr\"odinger
equation  \eqref{eqs5}. Now comparing  Schr\"odinger equations  \eqref{eqs5} and \eqref{eqschr} we observe that we have 
the nonlocal term $ u E(|u|^2)$ to treat. The main ingredient to do that will be an interpolation theorem and the 
generalization of the  Strichartz estimates for the Schr\"odinger equation derivated in \cite{CVeVi}.
As a consequence, we prove  existence and uniqueness (in the sense of distributions) to the IVP problem \eqref{eqs5}. 
As a consequence we find self-similar solutions for the problem  \eqref{eqs5} in the case $\delta>0$.

To study the IVP  \eqref{eqs5} we use its integral equivalent formulation
%We consider the Schr\"odinger equation in the Cauchy problem in its integral form

\begin{equation} \label{eqintds}
u(t) = U(t) u_0 +i\int_{0}^{t}U(t-s)(\chi|u|^{\alpha} u +b u E(|u|^{\alpha}))(s)ds,
\end{equation}

where $U(t)u_0$ defined as
\begin{align}
\widehat{U(t)u_0}(\xi)&=e^{-it\psi(\xi)} \widehat{u_0}(\xi),\label{defU}\\
\psi(\xi)&=4\pi^2\delta \xi_1^2+4\pi^2 \sum\limits_{j=2}^{n}{\xi_j^2},\nonumber
\end{align}

is the solution of the linear problem associated to \eqref{eqs5}.

%\begin{equation}\label{eqslinear}
%\left\{\begin{array}{l}
%i \partial_t u + \delta\partial_{x_1}^2 u +\sum\limits_{j=2}^{n}{\partial_{x_j}^2u}  =  0, \\
% u(x,0)  = u_0(x).
%\end{array}
%\right. \qquad\qquad \forall\, x \in \re^n ,\,t \in \re,
%\end{equation}

We also define the subspace $Y \subset S'(\re^n)$ where:
$$Y=\{\varphi \in S'(\re^n):U(t)\varphi \in L^{\frac{\alpha(n+2)}{2}\infty}(\re^{n+1})\},
$$ $$\|\varphi\|_Y=\|U(t)\varphi\|_{L^{ \frac{\alpha(n+2)}{2}\infty}(\re^{n+1})},$$
and $$L^{\frac{\alpha(n+2)}{2}\infty}(\re^{n+1})$$ are weak $L^p$ spaces that we define latter.
%The method of proof will be a combination of estimates and the contraction mapping principle.

Our main result in this paper reads as follows:

\begin{theorem}\label{TA}
There exists  $\delta_1>0$ such that given $\frac{4(n+1)}{n(n+2)}<\alpha<\frac{4(n+1)}{n^2}$ and  $u_0 \in Y$ with 
$\|u_0\|_Y<\delta_1 $ 
 then there exists
 a unique solution $u \in L^{\frac{\alpha(n+2)}{2} \infty}(\re^{n+1})$  of \eqref{eqintds} such that
$\|u\|_{L^{ \frac{\alpha(n+2)}{2}\infty}(\re^{n+1})}<3\delta_1$. 
 \end{theorem}

To obtain this result we will use the contraction mapping theorem and some estimates for the nonlocal operator
$E$, defined in \eqref{defE1}.

As a consequence of Theorem \ref{TA}  we show that giving any initial data in $Y$ and assuming the existence of a 
solution $u$ to the integral
 equation \eqref{eqintds} we have that $u$ is the solution (in the weak sense) of the differential equation \eqref{eqs5}.
We emphasize that Theorem \ref{TA} provides the existence of solutions to the equation \eqref{eqintds} under the assumption
of small initial data.

\begin{proposition}\label{propositionC}
Given $\frac{4(n+1)}{n(n+2)}<\alpha<\frac{4(n+1)}{n^2}$, $u_0 \in Y \text{ and let } u \in
 {L^{\frac{\alpha(n+2)}{2}\infty}(\re^{n+1})}$ be the solution of \eqref{eqintds}. It follows that $t \in \re
 \rightarrow u(t) \in S'(\re^n)$ is continuous and $u(0)=u_0$. In particular, $u$ is a solution of \eqref{eqs5}. Moreover 
$u(t_0)
 \in Y \ \text{for all} \  t_0 \in \re.$
In addition, there exist $u_{\pm}$ such that $\|U(t)u_{\pm}\|_{L^{ \frac{\alpha(n+2)}{2}\infty}(\re^{n+1})}<\infty $ and 
$U(-t)u(t)
\rightarrow u_{\pm}$ in $S'(\re^n)$ as $t\rightarrow
\pm \infty.$
\end{proposition}
This paper is organized as follows: in section \ref{global} we show the main theorem. In preparation for that we will establish
some needed estimates for the integral operator.

Last section will be devoted to find self-similar solutions.

\section{Global existence in weak $L^p$ spaces}\label{global}

Let us define the weak $L^p$ spaces we will use in the following:
%The next spaces  were introduced by Lorentz (\cite{L1}, \cite{L2}) and generalizes the $L^p$ spaces:
\begin{definition}\label{def1}
%The Lorentz space $L^{p q}(\re^n),1 \leq p,q \leq \infty$, is defined as follows:\\ $L^{p q}(\re^n)=\{f :\re^n 
%\rightarrow \mathbb{C} $ measurable $;\|f\|_{L^{pq}(\re^{n})}:= \left(\int_{0}^{\infty}{(t^{1/p}f^*(t))}^q \frac{1}
%{t}dt\right)^{1/q}<\infty \}$
% for $ q<\infty$ \\and\\
$L^{p \infty}(\re^n)=\{f :\re^n \rightarrow \mathbb{C} $ measurable $;\|f\|_{L^{p \infty}(\re^{n})} :=\sup\limits_
{\lambda>0}\lambda{\alpha(\lambda,
 f)}^{1/p}< \infty \}$\\
where
%$$f^*(t)=\inf\limits_{\lambda>0}\{{\alpha(\lambda, f)}\leq t\},$$
$$\ {\alpha(\lambda, f)}=\mu(\{x\in{\re}^n;|f(x)|>\lambda\}),$$and
$$\mu=\text{Lebesgue measure}.$$
\end{definition}

%The function $\alpha(\lambda, f)$ is called distribution function.

%In Chapter \ref{daveystew} the $L^{p\infty}$ spaces will be particularly relevant in our analysis. They are also called
% weak $L^p$ spaces.
The reader should  refer to \cite{BeL} for details.
\begin{remark}
Using a  change of variables it is easy to see that for any $ \ \varphi \in S'(\re^n),\ 1\leq p\leq \infty$ and $\tau \in \re$:
\begin{equation*}\label{idA}
 \|U(t)\varphi\|_{L^{p \infty}(\re^{n+1})}  = \|U(t+\tau)\varphi\|_{L^{p \infty}(\re^{n+1})},
\end{equation*}
where $U(t)$ is the unitary group defined in \eqref{defU}.
\end{remark}
The next theorem establishes a relationship between Lorentz Spaces $L^{p\infty}$ and $L^q$ spaces:

\begin{theorem}[Interpolation's theorem] \label{teoint} 
Given $0<p_0<p_1 \leq \infty$, then for all $ p$ and $\theta$ such that  $\frac{1}{p}=
\frac{1-\theta}{p_0}+\frac{\theta}{p_1}$ and $0<\theta<1$ we have :
$$(L^{p_0},L^{p_1})_{\theta,\infty}=L^{p\infty} \ 
\text{ with }\  \|f\|_{(L^{p_0},L^{p_1})_{\theta,\infty}}=\|f\|_{L^{p\infty}},$$
where

$$(L^{p_0},L^{p_1})_{\theta,\infty}=\{a \text{ Lebesgue measurable};\|a\|_{(L^{p_0},L^{p_1})_{\theta,\infty}}:=
\sup\limits_{t>0}t^{-\theta}k(t,a)<\infty\}$$
and
$$k(t,a)=\inf\limits_{a=a_0+a_1}(\|a_0\|_{L^{p_0}}+t\|a_1\|_{L^{p_1}}).$$
\end{theorem}

\begin{proof}
 We refer to \cite{BeL} for a proof of this theorem.
\end{proof}
\begin{remark}\label{mergulho}
Another relationship between Lorentz Spaces and $L^p$ spaces is given by the following decomposition:

Let $1\leq p_1<p<p_2<\infty$. Then
\begin{equation*}\label{dlp}
L^{p\infty}=L^{p_1}+L^{p_2}.
\end{equation*}
\end{remark}
The next result is a generalization of the classical Strichartz estimates for the Schr\"odinger equation. This was proved by Vilela in  \cite{McV}.
\begin{theorem}\label{teost}
 Consider $r , \tilde{r}, q$ and $\tilde{q} $ such that 
$$ 2 < r,  \tilde{r} \leq \infty,\ \ \dfrac{1}{{\tilde{r}}^{'}}-\dfrac{1}{r}<\dfrac{2}{n},$$
\begin{equation}\label{des1teost}
 \dfrac{1}{{\tilde{q}}^{'}}-\dfrac{1}{q}+\dfrac{n}{2}(\dfrac{1}{{\tilde{r}}^{'}}-\dfrac{1}{r})=1,
\end{equation}

\begin{equation}\label{des2teost}
\left\{\begin{array}{rclcc}
 r, \tilde{r} & \neq & \infty & \text{ if } & n=2, \\
\dfrac{n-2}{n}(1-\dfrac{1}{{\tilde{r}}^{'}}) & \leq & \dfrac{1}{r} \leq  \dfrac{n}{n-2}(1-\dfrac{1}{{\tilde{r}}^{'}})
& \text{ if }  & n\geq 3,
\end{array}
\right.
\end{equation}
and
\begin{equation*}
\left\{\begin{array}{rcl}
0< \dfrac{1}{q} \leq  \dfrac{1}{{\tilde{q}}^{'}} <1-\dfrac{n}{2}( \dfrac{1}{{\tilde{r}}^{'}}+\dfrac{1}{r}-1) 
\quad \text{ if } \quad
 \dfrac{1}{{\tilde{r}}^{'}}+\dfrac{1}{r}\geq 1, \\
 -\dfrac{n}{2}(\dfrac{1}{{\tilde{r}}^{'}}+\dfrac{1}{r}-1)<
\dfrac{1}{q}\leq\dfrac{1}{{\tilde{q}}^{'}}<1 
\quad \text{ if }  \quad \dfrac{1}{{\tilde{r}}^{'}}+\dfrac{1}{r}< 1.
\end{array}
\right.
\end{equation*}
\vspace{0.1cm}

Then we have the following inequalities:
\begin{align}
 \|\int_0^te^{i(t-\tau)\Delta}F(\cdot, \tau)d\tau\|_{L^q_tL^r_x}\leq c \|F\|_
{L^{{\tilde{q}}^{'}}_tL^{{\tilde{r}}^{'}}_x}\label{des3teost},
\end{align}
\begin{align*}
\|\int_{-\infty}^te^{i(t-\tau)\Delta}F(\cdot, \tau)d\tau\|_{L^q_tL^r_x}\leq c 
\|F\|_{L^{{\tilde{q}}^{'}}_tL^{{\tilde{r}}^{'}}_x},
\end{align*}
\begin{align*}
\|\int_{-\infty}^{+\infty}e^{i(t-\tau)\Delta}F(\cdot, \tau)d\tau\|_{L^q_tL^r_x}\leq c \|F\|
_{L^{{\tilde{q}}^{'}}_tL^{{\tilde{r}}^{'}}_x}\label{des4teost}.
\end{align*}
\end{theorem}

 \begin{proof}
  We refer to \cite{McV} for a proof of this theorem.
 \end{proof}
 \begin{remark}
  Theorem \ref{teost} also holds for $U(t)$.
\end{remark}

To prove Theorem \ref{TA} we need some results:

\begin{proposition}\label{prop2}
Consider $F:\re^n_x \times \re_t \rightarrow \ce$. Then for $1<p<\infty:$
\begin{equation*}\label{dmult}
  \|E(F)\|_{L^{p \infty}(\re^{n+1})} \leq  \| F\|_{L^{p \infty}(\re^{n+1})}.
\end{equation*}
\end{proposition}
Instead of proving Proposition \ref{prop2}
 we establish a more general 
result:

\begin{lemma}\label{lemma2}
Let  $A$ be a linear injective operator and suppose that for each  $1\leq p<\infty$ there exists $1\leq q=q(p)<\infty$ 
such that
 $\ A:L^p(\re^n)\rightarrow L^q(\re^n)$ is bounded. Then $A$
is bounded from $L^{p \infty}(\re^n)$ to $L^{q \infty}(\re^n)$.
\end{lemma}

\begin{proof}

In fact, fix $1\leq p <\infty$. Take $1\leq p_0,\ p_1<\infty$ and $0<\theta<1$ such that $\frac{1}{p}=\frac{1-\theta}{p_0}+
\frac{\theta}{p_1}$.
 By Theorem \ref{teoint} we have $\|A(f)\|_{L^{p \infty}(\re^{n})}=
\|A(f)\|_{(L^{p_0},L^{p_1})_{\theta \infty}}$. 

If $$f=f_0+f_1 \in L^{p_0}(\re^{n})+L^{p_1}(\re^{n}),$$
 then 
$$A(f)=A(f_0)+A(f_1) \in L^{{q_0}}(\re^{n})+L^{{q_1}}(\re^{n}),$$
 and
 $$ \|A(f_j)\|_{L^{q_j}(\re^{n})} \leq \| f_j\|_{L^{p_j}(\re^{n})},\ j=0,\ 1.$$ 
So 
\begin{align*}
  K(t,A(f))=&\inf\limits_{A(f)=F_0+F_1}(\|F_0\|_{L^{q_0}(\re^{n})}
+t\|F_1\|_{L^{q_1}(\re^{n})})\\
 \leq& \inf\limits_{A(f)=A(f_0)+A(f_1)}(\|A(f_0)\|_{L^{q_0}(\re^{n})}
+t\|A(f_1)\|_{L^{q_1}(\re^{n})})\\
\leq & \inf\limits_{A(f)=A(f_0)+A(f_1)}(\|f_0\|_{L^{p_0}(\re^{n})}+t\|f_1\|_{L^{p_1}(\re^{n})}).
\end{align*}
Since $A$ is injective, $A(f)=A(f_0)+A(f_1)\Rightarrow f=f_0+f_1 $ Lebesgue almost everywhere.

Then
$$ K(t,A(f)) \leq \inf\limits_{ f=f_0+f_1}(\|f_0\|_{L^{p_0}(\re^{n})}+t\|f_1\|_{L^{p_1}(\re^{n})})=K(t,f).$$
Using Theorem \ref{teoint} once more we obtain the result.
\end{proof}
Observe that since the linear  operator $E$  defined in \eqref{defE1} is injective and  satisfies
\begin{equation*}\label{Eqqforte}
 \|E(F)\|_{L^{q}(\re^{n+1})} \leq  \| F\|_{L^{q}(\re^{n+1})}
\end{equation*}
 for all $1<p<\infty$ (see \cite{X}), the Proposition \ref{prop2} will be a consequence of Lemma \ref{lemma2}.

Now we define two integral operators: 
\begin{equation}\label{defG}
 G(F)(x, t)=\int_{0}^{t}U(t-s)F(\cdot,s)(x)ds,
\end{equation}
and
\begin{equation}\label{defTT^*}
 (TT^*F)(x,t)=\int_{-\infty}^{+\infty}U(t-\tau)F(x, \tau)d\tau,
\end{equation}
where $U(t)$ is the group defined in \eqref{defU}. We prove the following properties about them:

\begin{proposition}\label{prop3}
Let $1\leq p,\ r< \infty$ such that
$$\frac{1}{p}-\frac{1}{r}=\frac{2}{n+2},$$
 and
$$ \frac{2(n+1)}{n}<r<\frac{2(n+1)(n+2)}{n^2}.$$
Then 
\begin{align}
  \|G(F)\|_{L^{r \infty}(\re^{n+1})} \leq c\|F\|_{L^{p \infty}(\re^{n+1})}\label{dnl},
\end{align}
and
\begin{align}
 \|TT^*(F)\|_{L^{r \infty}(\re^{n+1})}\leq c\|F\|_{L^{p \infty}(\re^{n+1})}\label{dTT^*}.
\end{align}
\end{proposition}
\begin{proof}
To prove Property \eqref{dnl} we need Theorem \ref{teost} (with $U(t)$ instead
 of $e^{it\Delta}$) and the interpolation theorem.
In fact taking $r=q$ and ${\tilde{r}}^{'}={\tilde{q}}^{'}=:p$ in Theorem \ref{teost}, the
hypothesis \eqref{des1teost} becomes
 $$\frac{1}{p}-\frac{1}{r}=\frac{2}{n+2},$$
and the inequality \eqref{des3teost} becomes
\begin{align}\label{k1}
 \|G(F)\|_{L^{r}(\re^{n+1})} \leq c\|F\|_{L^{p}(\re^{n+1})},
\end{align}

The restriction $\frac{2(n+1)}{n}<r<\frac{2(n+1)(n+2)}{n^2}$
comes from hypothesis \eqref{des2teost}.

The result follows applying Lemma \ref{lemma2} to inequality \eqref{k1}.
Property \eqref{dTT^*} is proved exactly the same way.
\end{proof}
Now we are ready to prove the our main result:

\begin{proof}[Proof of Theorem \ref{TA}:]
Consider the following operator 

\begin{equation}\label{defphi}
(\Phi u)(t)=U(t) u_0 -iG(\chi|u|^\alpha u +b u E(|u|^\alpha))(t),
\end{equation}
$G$ as in \eqref{defG}.
We want to use the Picard fixed point theorem to find a solution of $u=\Phi(u)$ in
$$ \overline{B(0, 3\delta_1)} =\{f \in L^{\frac{\alpha(n+2)}{2} \infty}(\re^{n+1});\|f\|_{L^{ \frac{\alpha(n+2)}{2}\infty}
(\re^{n+1})}\leq 3\delta_1\}.$$

 To prove $\Phi(\overline{B(0,3\delta_1)} \subset \overline{B(0,3\delta_1)}$ take $u \in \overline{B(0,3\delta_1)}$.

Using the hypothesis $\|U(t)u_0\|_{L^{\frac{\alpha(n+2)}{2} \infty}(\re^{n+1})}<\delta_1$ and Proposition \ref{prop3} 
combined with the definition $\Phi(\cdot)$ in  \eqref{defphi} , we obtain
$${\|\Phi(u)\|}_{L^{\frac{\alpha(n+2)}{2} \infty}(\re^{n+1})}\leq2\left(\delta_1+
\||u|^\alpha u\|_{L^{\frac{\alpha(n+2)}{2(\alpha+1)} \infty}(\re^{n+1})} +\|b u E(|u|^\alpha)\|_{L^{\frac{\alpha(n+2)}
{2(\alpha+1)} 
\infty}(\re^{n+1})}\right).$$
Applying Proposition \ref{prop2}  and Holder's inequality we get
$${\|\Phi(u)\|}_{L^{\frac{\alpha(n+2)}{2} \infty}(\re^{n+1})}\leq
2\left(\delta_1+
 \|u\|_{L^{\frac{\alpha(n+2)}{2} \infty}
(\re^{n+1})}^{\alpha+1}
+|b| \| u\|_{L^{\frac{\alpha(n+2)}{2}
 \infty}(\re^{n+1})} \| u(t)\|_{L^{\frac{\alpha(n+2)}{2} \infty}(\re^{n+1})}^\alpha\right).$$

Using that $u \in\overline{B(0,3\delta_1)}$ and choosing $0<\delta_1\ll 1$ we have
$${\|\Phi(u)\|}_{L^{\frac{\alpha(n+2)}{2} \infty}(\re^{n+1})}\leq 2c\delta_1+4c(3\delta_1)^{\alpha+1} +4c|b|
(3\delta_1)^{\alpha+1}<3\delta_1.$$

Now we prove the contraction in $B(0,3\delta_1)$. Take $u,v \in B(0,3\delta_1)$:
$$\Phi(u)-\Phi(v)=iG(\chi(|v|^\alpha v-|u|^\alpha u))+iG\big(b(vE(|v|^\alpha)-uE(|u|^\alpha)\big).$$
By Proposition \ref{prop3} we get
\begin{align*}
 &{\|\Phi(u)-\Phi(v)\|}_{L^{\frac{\alpha(n+2)}{2} \infty}(\re^{n+1})} \\
\leq \ &2c\left({\|v(|v|^{\alpha}-|u|^{\alpha})\|}_{L^{\frac{\alpha(n+2)}{2(\alpha+1)} \infty}(\re^{n+1})}
+{\||u|^{\alpha}(u-v)\|}_{L^{\frac{\alpha(n+2)}{2(\alpha+1)} \infty}(\re^{n+1})}\right)\\
&+2c|b|\left({\|E(|v|^{\alpha})(v-u)\|}_{L^{\frac{\alpha(n+2)}{2(\alpha+1)} \infty}(\re^{n+1})}
+{\|u(E(|v|^{\alpha})-E(|u|^{\alpha}))\|}_{L^{\frac{\alpha(n+2)}{2(\alpha+1)} \infty}(\re^{n+1})}\right).
\end{align*}

Applying Holder's inequality and Proposition \ref{prop2} we obtain
\begin{align*}
 &{\|\Phi(u)-\Phi(v)\|}_{L^{\frac{\alpha(n+2)}{2} \infty}(\re^{n+1})} \\
\leq \ &2c\bigg({\|v\|}_{L^{\frac{\alpha(n+2)}{2} \infty}(\re^{n+1})}{\||v|^\alpha-|u|^\alpha\|}_{L^{\frac{(n+2)}{2} \infty}
(\re^{n+1})}
+{\|u\|}^{\alpha}_{L^{\frac{\alpha(n+2)}{2} \infty}(\re^{n+1})}{\|u-v\|}_{L^{\frac{\alpha(n+2)}{2} \infty}
(\re^{n+1})}\bigg)\\
&+2c|b|\bigg({\|v\|}^{\alpha}_{L^{\frac{\alpha(n+2)}{2} \infty}(\re^{n+1})}{\|u-v\|}_{L^{\frac{\alpha(n+2)}{2} \infty}
(\re^{n+1})}
+{\|u\|}_{L^{\frac{\alpha(n+2)}{2} \infty}(\re^{n+1})}{\||v|^{\alpha}-|u|^{\alpha}\|}_{L^{\frac{(n+2)}{2} \infty}
(\re^{n+1})}\bigg).
\end{align*}
Now we set
$$g(u)=|u|^{\alpha}.$$ 
It follows by the Mean Value Theorem that 
\begin{equation*}
 |g(u)-g(v)|\leq c(\alpha)(|u|^{\alpha-1}+|v|^{\alpha-1})|u-v|.
\end{equation*}
This property and Holder's inequality imply that
\begin{align*}
&{\||v|^{\alpha}-|u|^{\alpha}\|}_{L^{\frac{(n+2)}{2} \infty}(\re^{n+1})}\\
\leq \  & c(\alpha)\left({\|u\|}^{\alpha-1}_{L^{\frac{\alpha(n+2)}{2} \infty}(\re^{n+1})}{\|u-v\|}_
{L^{\frac{\alpha(n+2)}{2} \infty}(\re^{n+1})}
+{\|v\|}^{\alpha-1}_{L^{\frac{\alpha(n+2)}{2} \infty}(\re^{n+1})} {\|u-v\|}_{L^{\frac{\alpha(n+2)}{2} \infty}(\re^{n+1})}
\right).
\end{align*}
Finally by the last inequality and the hypothesis $u,\ v \in B(0,3\delta_1)$ we get

$${\|\Phi(u)-\Phi(v)\|}_{L^{\frac{\alpha(n+2)}{2} \infty}(\re^{n+1})}\leq \delta_1^\alpha(c_1+c_2|b|){\|v-u\|}
_{L^{\frac{\alpha(n+2)}{2} 
\infty}(\re^{n+1})}.$$
Again taking $0< \delta_1\ll 1$ we get the contraction.
\end{proof}

 \begin{proof}[Proof of Proposition \ref{propositionC}:]
 By hypothesis $u \in L^{\frac{\alpha(n+2)}{2}\infty}(\re^{n+1})$. So by Holder's inequality and
 Proposition \ref{prop2}
 $$ |u|^\alpha u \ \ \text{ and }\ \  uE( |u|^\alpha )\in L^{\frac{\alpha(n+2)}{2(\alpha+1)}\infty}(\re^{n+1}).$$
 
  Now, by Remark \ref{mergulho} we can write 
 \begin{equation}\label{dec1}
 |u|^\alpha u=f_1+f_2 \ \  \text{ and } \ \ uE( |u|^\alpha )=f_3+f_4,
 \end{equation}
 
 where $f_j \in L^{p_j}(\re^{n+1})$  for some $1\leq p_1<\frac
{ \alpha(n+2)}{2(\alpha+1)}<p_2<\infty$ and   $1\leq p_3<\frac
{ \alpha(n+2)}{2(\alpha+1)}<p_4<\infty.$ 

Replacing \eqref{dec1} in \eqref{eqintds} we get 
\begin{equation}\label{dec3}
u(t)=U(t)u_0+i\chi G(f_1)(t)+i\chi G(f_2)(t)+ibG(f_3)(t)+ibG(f_4)(t).
\end{equation}
Observe that from the decomposition \eqref{dec3} we have that $u(t) \in S'(\re^n)$.

Now, if we take $\phi \in S(\re^n)$ then $U(t)\phi \in C(\re:S(\re^n))$ and also $G(\phi)(t) \in C(\re: S(\re^n))$. 
By duality we can 
extend $U(t)$ to $S'(\re^n)$ and get $U(t)\phi \in C(\re:S'(\re^n))$ for $\phi \in S'(\re^n)$. 

Using dominated convergence theorem  we have $G(\phi)(t) \in C(\re:S'(\re^n))$ for $\phi \in S'(\re^n)$ and by \eqref{dec3}
\begin{equation}\label{ucontinua}
u(t) \in  C(\re: S'(\re^n)). 
\end{equation}

 Letting $t\rightarrow 0$ in \eqref{dec3} we get
 $u(0)=u_0$.

Now we prove that $u(t)$ satisfies the equation
\begin{equation}\label{solucao}
i u_t + \delta u_{x_1x_1} +\sum_{j=2}^{n}{u_{x_jx_j}}=\chi |u|^\alpha u +bu E(|u|^\alpha),
\end{equation}
in $S'(\re^n)$ for all $t \in \re:$

Define $F(u):=\chi |u|^\alpha u +bu E(|u|^\alpha)$. 

Note that by \eqref{dec1} and \eqref{ucontinua} we have
\begin{equation*}\label{Fu}
 F(u)(t) \in C(\re, S'(\re^n)).
\end{equation*}

Using the integral equation \eqref{eqintds} and the definition of the operator $G$ in \eqref{defG} we have
the following expression for $u(t)$
\begin{equation}\label{v1}
 u(t)=U(t)u_0+iG(Fu)(t).
\end{equation}

Using group properties, Lebesgue dominated convergence theorem  and  the
Lebesgue  differentiation theorem combined with the expression of $u(t)$ in \eqref{v1} we obtain   that for any $ \ 
\phi \in S(\re^n)$
\begin{equation*}\label{identidade}
 i\lim_{h \rightarrow 0}\langle\frac{u(t+h)-u(t)}{h}, \phi\rangle=\langle - (\delta\partial_{x_1x_1}
 +\sum_{j=2}^{n}\partial_{x_jx_j})u(t)+F(u)(t), \phi\rangle,
\end{equation*}
where $\langle f, g\rangle=\int_{\re^n}f(x)g(x)dx$, which proves \eqref{solucao}.

To prove $\|u(t_0)\|_Y<\infty $, take $r=\frac{\alpha(n+2)}{2}$
on the Inequality \eqref{dTT^*} of Proposition \ref{prop3}.

 Then we have
 \begin{equation}\label{desTT*}
 \|TT^*F\|_{L^{ \frac{\alpha(n+2)}{2}\infty}(\re^{n+1})}\leq c
\|F\|_{L^{ \frac{\alpha(n+2)}{2(\alpha+1)}\infty}(\re^{n+1})}.\\
 \end{equation}
From the last property and  identity \eqref{idA}, $\forall \ t_0 \in \re$ we get 
$$\|U(t)\int_{-\infty}^{+\infty} U(t_0-s)F(s)ds\|_{L^{ \frac{\alpha(n+2)}{2}\infty}(\re^{n+1})}\leq
\|F\|_{L^{ \frac{\alpha(n+2)}{2(\alpha+1)}\infty}(\re^{n+1})}.$$
Now taking $\chi_{(0,t_0)}F$ instead of $F$ in the last inequality we have

\begin{equation}\label{desemY}
\|U(t)G(F)(t_0)\|_{L^{ \frac{\alpha(n+2)}{2}\infty}(\re^{n+1})}\leq
\|F\|_{L^{ \frac{\alpha(n+2)}{2(\alpha+1)}\infty}(\re^{n+1})},
\end{equation}
where $G$ was defined in \eqref{defG}.

 Now taking  $t=t_0$ in the integral equation \eqref{eqintds}
and applying $U(t)$ we have
$$U(t)u(t_0)=U(t+t_0)u_0+iU(t)G(\chi|u|^\alpha u +buE(|u|^\alpha ))(t_0).$$
Combining property \eqref{idA}, inequality \eqref{desemY} and the same arguments as in Theorem \ref{TA} we obtain

\begin{align*}
 \|U(t)u&(t_0)\|_{L^{ \frac{\alpha(n+2)}{2}\infty}(\re^{n+1})}\\
\leq & 2\|U(t)u_0\|_{L^{ \frac{\alpha(n+2)}{2}\infty}(\re^{n+1})}+
4\|u\|_{L^{ \frac{\alpha(n+2)}{2}\infty}(\re^{n+1})}^{\alpha+1}
 +4|b|\|u\|_{L^{ \frac{\alpha(n+2)}{2}\infty}(\re^{n+1})}
\|u\|_{L^{ \frac{\alpha(n+2)}{2}\infty}(\re^{n+1})}^{\alpha}<\infty.
\end{align*}

Finally, to prove  the last statement of the theorem we set
\begin{equation*}\label{u_+}
 u_{+}=u_0 +i\int_{0}^{\infty}U(-\tau)(\chi|u|^\alpha u+buE(|u|^\alpha))(\tau)d\tau.
\end{equation*}

It follows from Inequalitie \eqref{desTT*} that:
$$\|U(t)u_+\|_{L^{ \frac{\alpha(n+2)}{2}\infty}(\re^{n+1})}\leq2\left(\|U(t)u_0\|_
{L^{ \frac{\alpha(n+2)}{2}\infty}(\re^{n+1})}
+\|(\chi|u|^\alpha u +buE(|u|^\alpha ))\|_{L^{ \frac{\alpha(n+2)}{2(\alpha +1)}\infty}(\re^{n+1})}\right)<\infty.$$

We deduce from the decompositions in \eqref{dec1}  that
$$U(-t)u(t)-u_+=\int_{t}^{\infty}U(-\tau)(\chi|u|^\alpha u+buE(|u|^\alpha))(\tau)d\tau\rightarrow0 \text{ in }
S'(\re^n) \text{ as } t\rightarrow \infty.$$
The result for $t\rightarrow -\infty$ is proved similarly.

 \end{proof} 

\section{Self-similar solutions}\label{selfsimilar}

In this section we find self-similar solutions to \eqref{eqs5} for $\delta>0.$
Without lost of generality we can suppose $\delta=1$, so our equation becomes:

\begin{equation}\label{eqs1}
\left\{\begin{array}{rcl}
i u_t + \Delta u & = &  \chi |u|^\alpha u +bu E(|u|^\alpha), \\
 u(x,0)  &=&  u_0(x).
\end{array}
\right. \qquad\qquad \forall\, x \in \re^n , n=2,3,\,t \in \re,
\end{equation}
We will need the following proposition:

\begin{proposition}\label{prop1}
 Let $\varphi(x)=|x|^{-p}$ where $ 0<\Re p<n.$ Then $e^{it\Delta}\varphi$ is given   
the explicit formula below for $x\neq0$ and $t>0$:
\begin{align*}
 e^{it\Delta}\varphi(x)=&|x|^{-p}\sum_{k=0}^{m}A_k(a, b)e^{k\pi i/2}\left(\frac{|x|^2}{4t}\right)^{-k}\\
&+|x|^{-p}A_{m+1}(a, b)\left(\frac{|x|^2}{4t}\right)^{-m-1}\dfrac{(m+1)e^{aki/2}}{\Gamma (m+2-b)}\\
&\quad \times
\int_{0}^{\infty}\int_{0}^{1}(1-s)^m\left(-i-\frac{4ts\tau}{|x|^2}\right)^{-a-m-1}e^{-\tau}\tau^{m+1-b}dsd\tau\\
&+e^{i|x|^2/4t}|x|^{-n+p}(4t)^{\frac{n}{2}-p}\sum_{k=0}^{l}B_k(b, a)e^{-(n+2k)\pi i/4}\left(\frac{|x|^2}{4t}\right)^{-k}\\
&+e^{i|x|^2/4t}|x|^{-n+p}(4t)^{\frac{n}{2}-p}B_{l+1}(b, a)\left(\frac{|x|^2}{4t}\right)^{-l-1}\dfrac{(l+1)e^{aki/2}}
{\Gamma (l+2-b)}\\
&\quad \times\int_{0}^{\infty}\int_{0}^{1}(1-s)^l\left(-i-\frac{4ts\tau}{|x|^2}\right)^{-b-l-1}e^{-\tau}\tau^{l+1-a}dsd\tau,
\end{align*}
where $a=p/2, b=(n-p)/2, m, l \in \N$ such that $m+2>\Re b$ and $l+2>\Re a$\\
and
$$A_k(a, b)=\dfrac{\Gamma(a+k)\Gamma(k+1-b)}{\Gamma(a)\Gamma(1-b)k!},\ B_k(b, a)=\dfrac{\Gamma(b+k)\Gamma(k+1-a)}
{\Gamma(a)\Gamma(1-a)k!}$$
where $\Gamma$ denotes the gamma function.
\end{proposition}
\begin{proof}
 We refer to \cite{CW1} for a proof of this proposition.
\end{proof}

We already know that a self-similar solution must have an homogeneous initial condition with degree $-2/\alpha$.
So the idea is to prove that  $u_0(x)=\epsilon |x|^{-2/\alpha} \in Y$ where $0<\epsilon \ll 1 $.
Then by Theorem \ref{TA} and Proposition \ref{propositionC} we have existence and uniqueness for equation
\eqref{eqs1} in $Y$. Since $u(x,t)$ and $\beta^{2/\alpha}u(\beta x, \beta^2 t)$ are both solutions, we
must have $u=u_\beta$ and therefore self-similar solutions in $Y$.

To prove that $u_0 \in Y$, we consider the homogeneous problem with initial condition

$u_0(x)=|x|^{-2/\alpha}$:
\begin{equation}\label{eqs2}
\left\{\begin{array}{rcl}
i u_t + \Delta u & = & 0 ,\\
 u(x,0)  &=& |x|^{-2/\alpha}. 
\end{array}
\right. \qquad\qquad \forall\, x \in \re^n , n=2,3,\,t \in \re.
\end{equation}
We know that the solution to the equation \eqref{eqs2} is given by
$$u(x, t)=U(t)u_0(x),$$
 where $ U(t)=e^{it\Delta}.$

Since $u_\beta(x, t)=\beta^{2/\alpha}u(\beta x, \beta^2 t),\ \beta>0$ 
is also a solution, we must have
$$\beta^{2/\alpha}u(\beta x, \beta^2 t)=U(t)u_0(x)=u(x, t).$$
Taking $\beta=1/\sqrt{t}$ we get 
\begin{equation}\label{u}
 u(x,t)=t^{-1/\alpha}f(x/\sqrt{t}),
\end{equation}
where $f(x)=u(x, 1).$

By Proposition \ref{prop1} we have that for $\alpha>2/n$
\begin{equation}\label{f}
 |f(x)| \leq c(1+|x|)^{-\sigma} \text{ where }\sigma=\left\{\begin{array}{rcl}
 2/\alpha ;& \alpha\geq 4/n\\
 n-2/\alpha;& \alpha<4/n.
\end{array}
\right.
\end{equation}
Next, we calculate $\alpha(\lambda, u)=|\{(x, t); |u(x, t)|>\lambda\}|$.\\
By  \eqref{u} and \eqref{f} 
\begin{align*}
 \alpha(\lambda, u) \leq& \int_{\{(x, t); |t^{-1/\alpha}\left(1+\frac{|x|}{\sqrt{t}}\right)^{-\sigma}|>\lambda\}}d(x,t)\leq
\int_{\{(x, t); 0\leq t<\lambda^{-\alpha} \text{ and }
 |x|<t^{1/2}[(t\lambda^\alpha)^{-1/\alpha\sigma}-1]\}}d(x,t)\\
\leq & c \lambda^{-n/}\int_{0}^{\lambda^{-\alpha}}t^{\frac{n}{2}-\frac{n}{\sigma\alpha}}[1-(t\lambda^\alpha)^{\frac{1}
{\sigma\alpha}}]^n dt\leq \lambda^{\frac{-\alpha(n+2)}{2}}.
\end{align*}
Therefore $\|U(\cdot)u_0\|_{L^{ \frac{\alpha(n+2)}{2}\infty}(\re^{n+1})}\leq c.$\\
Choosing $0<\epsilon\ll 1$ and taking the initial condition $u_0(x)=\epsilon|x|^{-2/\alpha}$ we 
conclude the result.
\subsection*{Aknowledgements}I would like to thank my advisor Professor Felipe Linares for his help, his comments and all the fruitful
 discussions we had during the preparation of this work.


\begin{thebibliography}{article}
\addcontentsline{toc}{chapter}{Bibliography}
\bibitem[AF]{AF} J.M. Ablowitz, A.S. Fokas, On the inverse scattering transform of multidimensional
nonlinear equations, J. Math, Phys., Vol.{25} (1984), 2494-2505.
\bibitem[AnFr]{AnFr} D. Anker, N.C. Freeman, On the soliton solutions of the Davey-Stewartson equation for long waves
, Proc. R. Soc. A , Vol.{360} (1978), 529-540. 
\bibitem[AH]{AH} J.M. Ablowitz, R. Haberman, Nonlinear evolution equations in two and three dimensions, Phys. Rev. Lett.,
 Vol.{35} (1975),
 1185-1188.
\bibitem[AS]{AS} J.M. Ablowitz, A. Segur, Solitons and Inverse Scattering Transform, PA:SIAM (1981).
%\bibitem[AF]{AF} M.J. Ablowitz, A.S. Fokas, On the inverse scattering transform of multidimensional nonlinear equations
% related to first-order
 %systems in the plane, J. Math. Phys. \textbf{5} (1984), 2494-2505.
%\bibitem[BC]{BC} R. Beals, R.R. Coifman, The spectral problem for the Davey-Stewartson and Ishimori hierarchies, Proc. 
%Conf. on Nonlinear
 %Evolution Equations: Integrability and Spectral Methods, Machester, U. K., 1988.
 \bibitem[BeL]{BeL} J. Bergh, J. Lofstrom, Interpolation Spaces. An introduction, Springer-Verlag, Berlin-Heidelberg-New
 York, 1976.
\bibitem[C]{C} H. Cornille, Solutions of the generalized nonlinear Schr\"odinger equation in two spatial dimensions,
 J. Math. Phys., Vol.{20} (1979), 199-209.
\bibitem[Ch]{Ch} H. Chihara, The initial value problem for the elliptic-hyperbolic Davey-Stewartson equation, J. Math.
 Kyoto Univ.,
Vol.{39} (1999), 41-66.
%\bibitem[CC]{CC} M. Colin, M. Colin, On a quasilinear Zakharov system describing laser-plasma interactions, Differential
 %Integral
%Equations \textbf{17} (2004), 297-330.
%\bibitem[CM]{CM} T. Colin, G. M\'etivier, Instabilities in Zakharov equations for lazer propagation in a plasma, Phase 
%space
%analysis of partial differential equations, Progr. Nonlinear Differential Equations Appl. \textbf{69} (2006), 297-330, 
%Birkh\"auser Boston, MA.
 \bibitem[CP]{CP} M. Cannone, F. Planchon, Self-similar solution for the Navier-Stokes equations in
 $\re^3$, Comm. PDE, Vol.{21} (1996), 179-193.
\bibitem[CVeVi]{CVeVi} T. Cazenave, L. Vega, M.C. Vilela, A note on the nonlinear Schr\"odinger equation in weak $L^p$
 spaces, Comm. Contemporary
 Math., Vol.{3} (2001), 153-162.
 \bibitem[CW1]{CW1} T. Cazenave, F.B. Weissler, Asymptotically self-similar global solutions of the nonlinear
 Schr\"odinger and heat
 equations, Math. Z., Vol.{228} (1998), 83-120.
\bibitem[CW2]{CW2} T. Cazenave, F.B. Weissler, More self-similar solutions of the nonlinear Schr\"odinger equation,
 NoDEA Nonlinear Differential
 Equations Appl., Vol.{5} (1998), 355-365.
\bibitem[DR]{DR} V.D. Djordjevic, L.G. Redekopp, On two-dimensional packets of capillary-gravity waves, J. Fluid Mech., 
Vol.{79} 
(1977), 703-714.
\bibitem[DS]{DS} A. Davey, K. Stewartson, On three dimensional packets of surface waves, Proc. Roy. London Soc.
A, Vol.{338} (1974), 101-110.
%\bibitem[EK]{EK} M. Escobedo, O. Kavian, Asymptotic behaviour of positive solutions of a non-linear heat
%equation, Houston J. Math. \textbf{14}, no 1 (1988), 39-50.
\bibitem[F]{F} G. Furioli, On the existence of self-similar solutions of the nonlinear Schr\"odinger equation with 
power nonlinearity between 1 and 2, Differential Integral Equations, Vol.{14} , no 10 (2001), 1259-1266.
%\bibitem[Fa]{Fa} A.V. Faminskii, The cauchy problem for the Zakharov-Kuznetsov equation. Differential Equations
 %\textbf{31}, no 6 (1995), 1002-1012.
\bibitem[FS]{FS} A.S. Fokas, P.M. Santini, Recursion operators and bi-Hamiltonian structures in multidimensions I, II,
 Commun. Math. Phys., Vol.{115}, no 3 (1988), 375-419, Vol.{116}, no 3 (1988), 449-474.
\bibitem[FSu]{FSu} A.S. Fokas, L.Y. Sung, On the solvability of the N-wave, Davey-Stewartson and Kadomtsev-Petviashvili
 equations, Inverse Problems, Vol.{8} (1992), 375-419.
 %\bibitem[G]{G} L. Grafakos, Classical and Modern Fourier Analysis, Pearson Education, Inc., Upper Saddle River, NJ,
 %2004, 931pp. 
\bibitem[GM]{GM} Y. Giga, T. Miyakawa, Navier-Stokes flow in $\re^3$ with measures as initial vorticity and 
Morrey spaces, Comm. Partial Differential Equations, Vol.{14} (1989), 673-708.
 \bibitem[GS]{GS} J.M. Ghidaglia, J.C. Saut, On the initial problem for the Davey-Stewartson systems, Nonlinearity, 
Vol.{3}
 (1990), 475-506.
%\bibitem[GV]{GV} J. Ginibre, G. Velo, Smoothing Properties and Retarded Estimates for some Dispersive Evolution 
%Equations, Comm. 
%Math. Phys. \textbf{144}, no 1 (1992), 163-188.
%\bibitem[GTV]{GTV} J. Ginibre, Y. Tsutsumi, G. Velo, On the cauchy problem for the Zakharov system, J. Funct. Anal.
 %\textbf{151}, 
%no 2 (1997), 384-436.
\bibitem[H1]{H1} N. Hayashi, Local existence in time of small solutions to the Davey-Stewartson system, Annales de
 l'I.H.P.
 Physique Theorique, Vol.{65} (1996), 313-366.
\bibitem[H2]{H2} N. Hayashi, Local existence in time of solutions to the elliptic-hyperbolic Davey-Stewartson system
 without 
smallness condition on the data, J.Analys\'e Math\'ematique, Vol.{73} (1997), 133-164. 
\bibitem[HH1]{HH1} N. Hayashi, H. Hirata, Global existence and asymptotic behaviour of small solutions to the elliptic-
hyperbolic
 Davey-Stewartson system, Nonlinearity, Vol.{9} (1996), 1387-1409.
\bibitem[HH2]{HH2} N. Hayashi, H. Hirata, Local existence in time of small solutions to the elliptic-hyperbolic
 Davey-Stewartson system
 in the usual Sobolev space, Proc. Edinburgh Math. Soc., Vol.{40} (1997), 563-581.
\bibitem[HS]{HS} N. Hayashi, J.C. Saut, Global existence of small solutions to the Davey-Stewartson and the Ishimori
 systems,
Differential Integral Equations, Vol.{8} (1995), 1657-1675.
%\bibitem[I]{I} C. Isnard, Introdu\c{c}\~ao \`a medida e integra\c{c}\~ao, Projeto Euclides, IMPA, 2007.
\bibitem[K]{K} M. Kawak, A semilinear heat equation with singular initial data, Proc. Royal Soc. Edinburgh Sect. A,
Vol.{128} (1998), 745-758.
 %\bibitem[KT]{KT} M. Keel, T. Tao, Endpoint Strichartz estimates,  Amer. J. Math.  \textbf{120},  no 5 (1998), 955--980.
\bibitem[KW]{KW} O. Kavian, F. Weissler, Finite energy self-similar solutions of a nonlinear wave equation, Comm. PDE,
 Vol.{15}
(1990), 1381-1420.
%\bibitem[KZ]{KZ} C.E. Kenig, S.N. Ziesler, Maximal function estimate with aplications to a modified Kadomtsev-Petviashvili
%equation, Comm. Pure Appl. Anal. \textbf{4} (2005), 45-91.
%\bibitem[KPV]{KPV} C.E. Kenig, G. Ponce, L. Vega, Well-posedness and scattering results for the generalizad Korteweg-de 
%Vries equation via the contraction principle, %Comm. Pure Appl. Math. \textbf{46} (1993), 527-620.
%\bibitem[L1]{L1} G.G. Lorentz, Some new function spaces, Ann. Math. \textbf{51}
%(1950), 37-55.
%\bibitem[L2]{L2} G.G. Lorentz, On the theory of spaces $\Lambda$, Pac. J. Math. \textbf{1} (1951), 411-429.
 \bibitem[LP1]{LP1} F. Linares, G. Ponce, On the Davey-Stewartson systems, Ann. Inst. Henry Poincar\'e, Vol.{10} 
no 5 (1993), 523-548. 
\bibitem[LP2]{LP2} F. Linares, G. Ponce, Introduction to Nonlinear Dispersive Equation, Springer, New York, 2009, 256pp.
%\bibitem[LiPoS]{LiPoS} F. Linares, G. Ponce, J.C. Saut, On a degenerate Zakharov System. Bull. Braz. Math. Soc.
%\textbf{36}, no 1 (2005), 1-23.
 \bibitem[McV]{McV} M.C. Vilela, Las estimaciones de Strichartz bilineales en el contexto de la ecuaci\'on de Schr\"odinger,
 Ph.D.thesis, Universidad del Pais Vasco (2003), Bilbao.
%\bibitem[O]{O} R. O`neil, Convolution Operators and $L^{p,q}$ Spaces, Duke Math. J. \textbf{30} (1963), 129-142.
 %\bibitem[Oh]{Oh} M. Ohta, Stability of standing waves for the generalized Davey-Stewartson system, J. Dynam.
 %Differential Equations  \textbf{6},  no 2 (1994), 325--334. 
%\bibitem[OT]{OT} T. Ozawa, Y. Tsutsumi, Existence and smoothing effect of solutions for the Zakharov equations,
%Publ. Res. Inst. Math. Sci. \textbf{28}, no 3 (1992)
% \bibitem[Oz]{Oz} T. Ozawa, Exact blow-up solutions to the Cauchy problem for the Davey-Stewartson systems,  Proc. Roy.
% Soc. London Ser. A  \textbf{436}  (1992),  no 1897, 345--349.
\bibitem[P]{P} F. Planchon, Self-Similar solutions and semilinear wave equations in Besov spaces, J. Math. Pures Appl.,
  Vol.{79}, no 8 (2000), 809-820.
\bibitem[Pe]{Pe} H. Pecher, Self-similar and asymptotically self-similar solutions of nonlinear wave equations, Math. Ann.,
Vol.{316} (2000), 259-281. 
%\bibitem[Ru]{Ru} W. Rudin, Real and Complex Analysis.Third edition, McGraw-Hill Book Co.,New York, 1987. 
\bibitem [RY1]{RY1} F. Ribaud, A. Youssfi, Regular and self-similar solutions of nonlinear Schr\"odinger equations, 
J. Math Pures Appl., Vol.{77} (1998), 1065-1079.
\bibitem[RY2]{RY2} F. Ribaud, A. Youssfi, Global solutions and self-similar solutions of semilinear wave equation, 
Math. Z., Vol.{239} (2002), 231-262.
\bibitem[Su]{Su} L.Y. Sung, An inverse-scattering transform for the Davey-Stewartson II equations, Part III, 
J. Math Anal. Appl., Vol.{183} (1994), 477-494.
\bibitem[STW]{STW} S. Snoussi, S. Tayachi, F.B. Weissler, Asymptoticaly self-similar global solutions of a general
 semilinear heat equation, Math. Ann., Vol.{321} (2001), 131-155.
 \bibitem[X]{X} Z. Xiangking, Self-Similar solutions to a generalized Davey-Stewartson system, Adv. Math. (China), 
Vol.{36}, no
 5 (2007), 579--585.
%\bibitem[Z]{Z} V.E. Zakharov, Collapse of Langmuir Waves, Sov. Phys. JETP \textbf{35} (1972), 908-914.
%\bibitem[ZK]{ZK}V.E. Zakharov, E.A. Kutnetsov, Hamilton formalism for systems of hydrodynamic type, Math. Physics Review, Sov. Sci. Rev. \textbf{4}
 % (1984), 167-220.
\end{thebibliography}
\end{document}